\documentclass{amsart}
\usepackage{fullpage, amsthm, amsfonts, amssymb, amscd, xfrac, graphicx, multicol}
\usepackage{stmaryrd}
\usepackage{framed}

\pdfoutput=1

\DeclareMathOperator{\hgt}{ht}

\renewcommand{\char}{\text{char}}

\newcommand{\mb}{\mathbb}

\renewcommand{\frak}{\mathfrak}

\newcommand{\la}{\langle}
\newcommand{\ra}{\rangle}
\newcommand{\sseq}{\subseteq}

\newcommand{\ses}[3]{
  \[\begin{CD}
    0 @>>> #1 @>>> #2 @>>> #3 @>>> 0
  \end{CD}\]
}

\newcommand{\m}{\frak{m}}

\newcommand{\NN}{\mb{N}}

\swapnumbers
\newtheorem{thm}{Theorem}[section]

\newtheorem{lem}[thm]{Lemma}

\theoremstyle{remark}

\newtheorem{exa}[thm]{Example}

\newtheorem{quest}[thm]{Question}

\theoremstyle{definition}
\newtheorem{defn}[thm]{Definition}

\newcommand{\frob}[1]{F^{#1}_*}
\newcommand{\fpow}[1]{^{[p^{#1}]}}

\DeclareMathOperator{\sdim}{sdim}

\title{Relating F-signature and F-splitting Ratio of Pairs\\Using Left-Derivatives}
\author{Eric Canton}
\address{Department of Mathematics\\ University of Nebraska - Lincoln\\ Lincoln,  NE 68588}
\email{ecanton2@math.unl.edu}

\begin{document}
\maketitle
\begin{abstract}
We first relate an approximate $n^{th}$-order left derivative of $s(R, f^t)$ at the F-pure threshold $c$ to the F-splitting ratio 
$r_F(R, f^c)$. Next, we apply the methods developed by Monsky and Teixeira in their investigation of syzygy gaps and $p$-fractals to obtain 
uniform convergence of the F-signature when $f$ is a product of distinct linear polynomials in two variables. Finally, we explicitly compute the F-signature function for several examples using Macaulay2 code outlined in the last section of this paper.
\end{abstract}

\section{Introduction}
Let $(R, \m, k)$ be a local ring containing a field $k$, and assume $\char(R) = p > 0$. For simplicity we assume that $k = k^p$ ($k$ is perfect); 
this is always the case when our field is finite. For an ideal $I \sseq R$ we denote by $I\fpow{e} = \{\sum_1^n c_i r_i^{p^e} : c_i \in R 
\text{ and }r_i \in I\}$ the ideal generated by all $p^e$-th powers of elements in $I$. Because $\char(R)$ is positive this is in general an 
ideal distinct from the normal $p^e$-th power of the ideal.

The {\bf Frobenius endomorphism} of $R$ is the map $F: R \to R$ which takes $r \mapsto r^p$. We will write $R$ for the domain of this map and 
$\frob{}R$ for the codomain, when it is important to distinguish the two (although they are isomorphic as abelian groups). We define an $R$-module
structure on $\frob{}R$ by $r.x = r^px$ for $r \in R$ and $x \in \frob{}R$. A reduced ring $R$ is {\bf F-finite} if $\frob{}R$ is finitely 
generated as an $R$-module. Similarly, we can consider the $e$-th iterated Frobenius map $F^e: R \to R$ and define an $R$-module structure for 
this iterated Frobenius map; the codomain is denoted here $\frob{e}R$. Note that if $R$ is F-finite, then $\frob{e}R$ is finitely generated for 
all $e \in \NN$.

Important classes of F-finite rings include polynomial rings over a field in positive characteristic, quotients, and localizations of these rings.
For example consider $S = k[x_1, \dots, x_n]$. Because $\frob{}S$ is generated as an $S$-module by the products $\Pi_{i=1}^n x_i^{d_i}$ with each 
$d_i \le p-1$, we conclude that $S$ is F-finite. Similarly, when $R$ is an F-finite ring and $I \sseq R$ is an ideal, then $\frob{}(R/I)$ is 
generated over $R/I$ by the images of the generators of $\frob{}R$ over $R$. Thus any quotient of an F-finite ring is also F-finite. We also have 
that localizations of F-finite rings are again F-finite (see R. Fedder, Lemmas 1.4 and 1.5 from \cite{Fedd83} for more information.)

When $R$ is reduced, we naturally identify $\frob{e}R$ with the ring $R^{1/p^e}$ of $p^e$-th roots of elements of $R$ by sending 
$r \mapsto r^{1/p^e}$. Next, we will decompose this module as $R^{1/p^e} = R^{a}\oplus M$ where $M$ has no free summands of $R$. We define $a_e$ to 
be the maximal rank of any free decomposition of $R^{1/p^e}$. A famous result of E. Kunz (which we refer to as Kunz's Theorem) states that for all 
$p^e$, we have that $a_e \le p^{ed}$ with equality if and only if $R$ is regular if and only if the Frobenius map is flat. This result prompted 
C. Huneke and G. Leuschke to define \cite{HL04} the F-signature as the following limit.

\begin{defn} Let $(R, \m, k)$ be an F-finite, $d$-dimensional reduced local ring and $k = R/\m$ a field of positive characteristic $p$. The 
{\bf F-signature} of $R$ is
\[
  s(R) := \lim_{e \to \infty} \frac{a_e}{p^{ed}}
\]
\end{defn}

Recently, K. Tucker showed that this limit exists in full generality \cite{FsigExists}. It is 1 if and only if $R$ is regular, and so the 
F-signature serves as a measure to which $R$ fails to be regular in comparison to other local rings of the same dimension. If for some $e$ (and 
thus every $e$) we have that $a_e \ne 0$ then we say that $R$ is {\bf F-pure}. Often it happens that the limit $s(R)$ is zero. I. Aberbach and 
F. Enescu defined \cite{AE03} the {\bf Frobenius splitting dimension} (F-splitting dimension) $\sdim(R) = m$ to be the greatest integer such that 
\begin{align*}
  \lim_{e \to \infty} \frac{a_e}{p^{em}}
\end{align*}
is greater than zero. This limit exists by \cite{FsigExists}. The corresponding limit is defined in \cite{AE03} to be the 
{\bf Frobenius splitting ratio} (F-splitting ratio) denoted as $r_F(R)$. Note that if the splitting dimension $\sdim(R) = \dim(R)$, then 
$r_F(R) = s(R)$.

In \cite{BST11} M. Blickle, K. Schwede, and K. Tucker introduce the concept of F-signature to the pair $(R, f^t)$, where $t \in [0, \infty)$ is a 
real number and $f \in R$ is nonzero. 
\begin{defn}
  Let $f\in R$ be a nonzero element in an F-finite regular local ring $(R, \m, k)$. Let $d = \dim(R)$ be the Krull dimension of $R$, and 
  $t \in [0, \infty)$ a positive real number. The {\bf F-signature} of the pair $(R, f^t)$ is defined to be the limit
  \[
     s(R, f^t) := \lim_{e\to\infty} \frac{1}{p^{ed}}\ell_R\left(\frac{R}{\m\fpow{e}:f^{\lceil t(p^e-1) \rceil}}\right)
  \]
\end{defn}
The supremum over all $t$ such that $s(R, f^t)$ is F-pure is called the {\bf F-pure threshold} of $f$ and will be denoted in this paper
as FPT$(f)$.

An important theorem regarding computation of the F-signature when $t = a/p^s$ is Proposition 4.1 found in \cite{BST112}, which states that 
\[
   s(R, f^{a/p^s}) = \frac{1}{p^{sd}}\ell_R\left(\frac{R}{\m^{[p^s]}:f^a}\right)
 \]
which is to say that in this case, we do not need to take a limit: a single length suffices to compute $s(R, f^{a/p^s})$.

In much the same way, we may now define the {\bf F-splitting dimension} of the pair $(R, f^t)$, where $t$ is a rational number whose denominator 
is not divisible by $p$, to be the greatest integer $m$ such that 
\[
  \limsup_{e\to\infty} \frac{1}{p^{em}}\ell_R\left(\frac{R}{\m\fpow{e}: f ^{\lceil t(p^e-1)\rceil}}\right)
\]
is nonzero. We again define this limit to be the {\bf Frobenius splitting ratio} (F-splitting ratio) of the pair $(R, f^t)$, denoted here 
$r_F(R, f^t)$. The first result in this paper relates an approximate higher-order left derivative of $s(R, f^t)$ at $c = FPT(f)$ to a constant 
multiple of $r_F(R, f^c)$ when $p$ does not divide the denominator of $c$. This is a generalization of Theorem 2.1 in the next section, which can 
be found in \cite{BST112} as Theorem 4.2 relating the first left derivative $D_-s(R, f^1)$ to the F-signature $s(R/\la f \ra)$.

Computing the F-splitting dimension is aided by formation of a special ideal, called the {\bf splitting prime} of $(R, f^t)$. This ideal is defined
 to be the maximal ideal $J \sseq R$ such that $f^{\lceil t(p^e-1)\rceil}J \sseq J\fpow{e}$ for all $e > 0$. As the name suggests, when it is a proper ideal of $R$ it is a prime ideal. This result can be found in (\cite{Sch08}, Corollary 6.4) but it is presented again here for the reader's convenience, albeit with a different proof.

\begin{lem}\label{splittingPrime} Let $(R, \m, k)$ be an F-finite regular local ring, and $f \in R$ a nonzero element of $R$. 
  Take $t \in [0, \infty)$ and let $P$ be the splitting prime of the pair $(R, f^t)$. If $P$ is proper, then it is a prime ideal.\end{lem}
\begin{proof}
  Suppose $P \ne R$ is the splitting prime of $(R, f^t)$ and let $c \in R \setminus P$. We wish to show that $(P:c) = P$ which implies that $P$ 
  is a prime ideal.

  I claim that if $J$ is an ideal of $R$ satisfying $f^{\lceil t(p^e-1)\rceil}J \sseq J\fpow{e}$, then for any $r \in R \setminus J$ we also have that
  $f^{\lceil t(p^e-1)\rceil}(J:r) \sseq (J\fpow{e}: r^{p^e})$. To see this, let $g \in (J:r)$ so that $gr \in J$. Then $f^{\lceil t(p^e-1)\rceil}gr \in
  J\fpow{e}$ by the assumption we made on $J$. Then of course $f^{\lceil t(p^e-1)\rceil}gr^{p^e} \in J\fpow{e}$, and so $f^{\lceil t(p^e-1)\rceil}g \in 
  (J\fpow{e}:r^{p^e})$, establishing the claim.

  By Kunz's theorem, we know that the Frobenius endomorphism on $R$ is flat when $R$ is regular, so we may tensor the exact sequence
  \ses{R/(J:r)}{R/J}{R/(J+r)} with $\frob{e}R$ to conclude that in this case $(J\fpow{e}:r^{p^e}) = (J:r)\fpow{e}$. Therefore, whenever $J$ is an 
  ideal satisfying $f^{\lceil t(p^e-1)\rceil}J\sseq J\fpow{e}$ and $r \in R \setminus J$ is any element, we have that 
  $f^{\lceil t(p^e-1)\rceil}(J:r)\sseq (J:r)\fpow{e}$.
  
  Because $P \ne R$ is the splitting prime for the pair $(R, f^t)$, it is contained in no other ideal satisfying $f^{\lceil t(p^e-1)\rceil}J
  \sseq J\fpow{e}$. This implies that $(P:c) = P$ and so $P$ is prime.
\end{proof}
The splitting prime is related to the F-splitting dimension $\sdim(R, f^t)$ by the following theorem, which can be found in \cite{BST11}.

\begin{thm}[\cite{BST11} Theorem 4.2] Let $(R, \m, k)$ be an F-finite $d$-dimensional regular local ring, $f \in R$ a nonzero element, 
  $t \in [0, \infty)$ and let $P$ be the splitting prime of the pair $(R, f^t)$. Then
  \[
  \sdim(R, f^t) = \dim(R/P)
  \]
\end{thm}

In the third section of this paper, we provide an application of Monsky and Teixeira's work on $p$-fractals to compute the F-signature when $f$ is
a homogeneous polynomial in two variables. Specifically, we prove that $s(R, f^t)$ converges uniformly to a quadratic polynomial in $t$ as 
$p \to \infty$. We provide an example where we can explicitly compute the F-pure threshold for $f = xy(x+y)$ and apply the $p$-fractal techniques 
mentioned before to compute the left derivative at the F-pure threshold. 

Finally, several computational examples (using Macaulay2) are included in the second-to-last section of this paper; algorithms that were used to 
compute these examples comprise the final section.

{\em Special thanks to Karl Schwede for many insightful and encouraging discussions over the course of this work and preparation of this document.
 In particular, the proofs of \ref{splittingPrime} and \ref{DerivativeThm} were discussed with him. I would like to thank Kevin Tucker for a 
stimulating discussion of some of the results presented here and his suggestions regarding this paper, and I would also like to thank Florian 
Enescu, who suggested applying the results of Teixeira's thesis to F-signature in two variables. Thanks to Wenliang Zhang and Lance Miller for 
their useful critiques of this paper.}


\section{F-Splitting Ratio of Principal Ideals}
In this section, we assume that $R$ is an F-finite regular local {\em domain} with Krull dimension $d$. The F-signature of the pair $(R, f^t)$ was
 shown recently in \cite{BST112} to be continuous and convex on $[0, \infty)$, thus differentiable almost everywhere on the domain $[0, \infty)$. 
Indeed, the authors of \cite{BST112} proved the following theorem:

\begin{thm}[\cite{BST112} Theorem 4.3] If $(R, \m, k)$ is an F-finite $d$-dimensional local domain and $f\in R$ is a nonzero element, then
\begin{align*}
  &D_-s(R, f^1) = -s(R/\la f \ra)   &D_+s(R, f^0) = -e_{HK}(R/\la f\ra)
\end{align*}
\end{thm}

\noindent In this section, we generalize the first part of the above result to an approximate left $n^{th}$-derivative at the F-pure threshold of 
our nonzero element $f$. Let $FPT(f) = c$ be the F-pure threshold of $f$, and denote by $n= d - \sdim(R, f^c)$. We make the following assumption 
on $c$ in this section only.

\begin{center}
\begin{tabular}{| c |}
\hline

Assume that $c$ is a rational number whose denominator is not divisible by $p$.\\

\hline
\end{tabular}
\end{center}

\noindent In particular, the pair $(R, f^c)$ is sharply F-pure \cite{Sch08} and so the splitting prime is proper. Because $p$ does not divide the
denominator of $c$, we can write $c = a/(p^s-1)$ for an appropriate $a$ and $s$. To arrive at an approximation of $c$, we write
\begin{align*}
  c &= \frac{a}{p^s-1}\\
  &= \frac{a(p^{(e-1)s} + p^{(e-2)s} + \cdots + 1)}{(p^s-1)(p^{(e-1)s} + p^{(e-2)s} + \cdots + 1)}\\
  &= \frac{a(p^{(e-1)s} + p^{(e-2)s} + \cdots + 1)}{p^{es}-1}
\end{align*}
Now let $K_e = (p^{(e-1)s} + p^{(e-2)s} + \cdots + 1)$, so that above we have $c = aK_e/(p^{es}-1)$. Define then 
\[
  t_e = \frac{aK_e}{p^{es}}
\]

\noindent and note that $t_e \to c$ as $e \to \infty$. We compute the following limit, which serves as a sort of approximate left 
$n^{th}$-derivative:
\begin{align*}
  \limsup_{t_e \to c} \frac{s(R, f^{t_e})}{(t_e-c)^n} &= \limsup_{e\to\infty}\left(-\frac{p^{es}(p^s-1)}{a}\right)^n\frac{1}{p^{esd}}\ell_R\left(\frac{R}{\m\fpow{es}:\la f\ra^{(p^{es}-1)c}}\right)\\
  &= \left(-\frac{p^s-1}{a}\right)^n \limsup_{e\to \infty}\frac{1}{p^{es(d-n)}}\ell_R\left(\frac{R}{\m\fpow{es}:\la f \ra^{aK_e}}\right)\\
  &= (-c)^{-n} r_F(R, f^c)
\end{align*}

Since $n$ is the smallest integer such that the above limit is nonzero, if $\sdim(R, f^t) < \dim(R) - 1$ then we have $n = \dim(R) - \sdim(R, f^t)
 \ge 2$. Because we know that $s(R, f^t)$ is differentiable almost everywhere, if $FPT(f) = c$ as above and we have $\sdim(R, f^c) < \dim(R) - 1$.
Therefore the left derivative $D_-s(R, f^c) = 0$. {\em The proof in the non-square-free case was suggested by Wenliang Zhang.}

\begin{thm}\label{DerivativeThm} Suppose $f \in R$ and $(R, f^c)$ is sharply F-pure for $c < 1$. Write $f = uf_1^{n_1}\cdots f_r^{n_r}$, where $u$ 
  is a unit and each $f_i$ is irreducible. If $f$ is not square free, assume that $FPT(f_i)=c_i<1$ for each $i$ such that $n_i \ge 2$. Then the 
  left derivative $D_-s(R, f^c) = 0$.
\end{thm}
\begin{proof} Let $P$ be the splitting prime of the pair $(R, f^c)$. I claim that with the hypotheses of the theorem, we have $\dim(R/P) < \dim(R)-1$. Because $(R, f^c)$ is sharply F-pure, this implies that $P$ is proper and so prime. Towards a contradiction assume that $\hgt(P) = 1$. We will
 need the following lemma:

\begin{lem}\label{localizationLemma} If $P$ is the splitting prime of $(R, f^c)$ then $PR_P$ is the splitting prime of $(R_P, f_P^c)$ where $f_P$ 
is the image of $f$ in $R_P$.\end{lem}
\begin{proof}[Proof of lemma \ref{localizationLemma}]
  $PR_P$ satisfies $f_P^{\lceil c(p^e-1) \rceil}(PR_P) \sseq (PR_P)\fpow{e}$ by definition of $P$. Also, prime ideals of $R_P$ are in bijective 
  correspondence with primes of $R$ contained in $P$; thus because $PR_P$ is maximal in $R_P$, it must be the splitting prime of $(R_P, f_P^c)$. 
\end{proof}

We now consider two cases: either $f$ is square free, or $FPT(f_i) = c_i < 1$ for all $i$ such that $n_i \ge 2$. Because $\hgt(P) = 1$, this 
implies that $P = \la f_i\ra$ for some $i$. Localize at $P$ and suppose that $f$ is square-free or $n_i = 1$. This tells us that $PR_P = fR_P$, 
since $R_P$ is a DVR with maximal ideal $PR_P$ and $f_P$ cannot be a unit, since $f_P = (uf_1^{n_1}\cdots f_i^1 \cdots f_r^{n_r})_{\la f_i \ra} = vf_i$,
 where $v \in R_P$ is a unit and the image of $f_i$ generates the maximal ideal. Note that by assumption that $c <1$ is a rational number whose 
denominator is not divisible by $p$, there exist infinitely many $e$ such that $c(p^e-1)$ is an integer $r_e$, and each such $r_e$ satisfies 
that $c(p^e-1) < r_e+1 < (p^e-1)+1$. This implies that
\begin{equation*}
  f_P^{r_e}PR_P \not\sseq (PR_P)\fpow{e} 
\end{equation*}
contradicting that $PR_P$ is the splitting prime of $R_P$. 

Suppose then that $P = \la f_i \ra$ and $n_i\ge 2$, and recall that by assumption $FPT(f_i) = c_i <1$. This gives that $c \le c_i/n_i < 1/n_i$, so 
$cn_i(p^e-1)+1 \le c_i(p^e-1)+1 < p^e$. The same argument as the previous case leads to a contradiction of lemma \ref{localizationLemma}. Thus, 
$\hgt(P) \ge 2$ and by Theorem 1.2, we have
\[
  \dim(R/P) = \sdim(R, f^c)
\]
so $\dim(R/P) < \dim(R)-1$, implying that $n \ge 2$ and so $D_-s(R, f^c) = 0$. 
\end{proof}

It follows immediately that because $s(R, f^r) = 0$ for all $r \ge c$, $D_+s(R, f^c) = 0$ and so the F-signature is differentiable at $c = FPT(f)$
 whenever $c<1$ is a rational number whose denominator is not divisible by $p$. In the next two sections, we will see examples where the result is
 false when $p$ divides the denominator of $c$.


\section{F-Signature of Homogeneous Polynomials in Two Variables}
We turn our attention now to the case when $R = k[x,y]_{\la x, y \ra}$ and let $f \in R$ be a product of $r\ge 2$ distinct linear forms with 
$FPT(f) = c$. Here we relax the condition of the previous section that if $c$ is rational, then $p$ does not divide the denominator. By the exact 
sequence
\begin{equation}\begin{CD}
  0 @>>> \dfrac{R}{\m\fpow{e}:f^a} @>>> \dfrac{R}{\m\fpow{e}} @>>> \dfrac{R}{\m\fpow{e}+\la f^a \ra} @>>> 0 
\end{CD}\label{sesFsigHomog}\end{equation}
we have that $s(R, f^{a/p^s}) = 1 - \frac{1}{p^{2s}}\ell_R(R/\la x^{p^s}, y^{p^s}, f^a\ra)$. This length has been studied extensively by P. Monsky and 
P. Teixeira in their work on $p$-fractals. We can use theorems found in \cite{Teix02} and \cite{Mon06} to obtain the following result: 

\begin{thm} The F-signature of the pair $s(R, f^t)$ where $f$ is the product of at least two distinct linear factors converges uniformly on the 
interval $[0, c]$ to the polynomial $\frac{r^2}{4}t^2 - rt + 1$ as $p \to \infty$.\end{thm}

\begin{proof}
The Hilbert Syzygy Theorem implies that the module of syzygies between $x^{p^e}, y^{p^e}$ and $f^a$ can be generated by two homogeneous elements of 
degrees $m_1 \ge m_2$. Their difference $\delta = m_1 - m_2$ is called the {\bf syzygy gap} of $(x^{p^e}, y^{p^e}, f^a)$. If we need to consider more
 than one triplet $(x^{p^e}, y^{p^e}, f^a)$ we will write this syzygy gap as $\delta(x^{p^e}, y^{p^e}, f^a)$. Theorem 2.10 in \cite{Teix02} tells us 
that
\[
   \ell_R(R/\la x^{p^e}, y^{p^e}, f^a\ra) = \frac{1}{4}(4rap^e - (ra)^2) + \frac{\delta^2}{4}
\]
Also in his thesis \cite{Teix02}, Teixeira showed the functions $\frac{a}{p^e} \mapsto \frac{1}{p^{2e}}\ell_R(R/\la x^{p^e}, y^{p^e}, f^a\ra)$ and 
$\frac{a}{p^e} \mapsto \frac{1}{p^e}\delta(x^{p^e}, y^{p^e}, f^a)$ defined on $[0, 1] \cap \mathbb{Z}[p^{-1}]$ can be extended uniquely to continuous
 functions on $[0, \infty)$. These extended functions are denoted $\phi_f(t)$ and $\delta_f(t)$ respectively. 

In \cite{Mon06} Monsky proved an upper bound for $\delta(x^{p^e}, y^{p^e}, f^a)$ in the case when $f$ is homogeneous of degree $\ge 2$. 

\begin{thm}[\cite{Mon06} Theorem 11] Let $l_1, \dots, l_r$ be linear forms such that $l_i$ and $l_j$ share no common non-unit factor for $i \ne j$
  and $r \ge 2$. Suppose $0 \le a_1, \dots, a_r \le p^e$ and he $a_i$ satisfy the inequalities $2a_i \le \sum_1^r a_j \le 2p^e$. Then 
  $\delta(x^{p^e}, y^{p^e}, \Pi_1^r l_i^{a_i}) \le (r-2)p^{e-1}$. 
\end{thm}

In our case, this theorem tells us that for $f$ the product of $r \ge 2$ distinct linear forms and $ra \le 2p^e$ then
\[
\delta(x^{p^e}, y^{p^e}, f^a) \le (r-2)p^{e-1}.
\]
Rearranging $ra \le 2p^e$, we see that if $(ra)/2 \le p^e$, then Monsky's bound holds. Note that each term in $f^a$ has degree in $x$ or degree in
 $y$ at least $(ra)/2$, so if $f^a \not\in \m\fpow{e}$ then $(ra)/2 \le p^e$. Now, we remember the exact sequence \eqref{sesFsigHomog} to compute
\begin{align*}
  s(R, f^{a/p^e}) &= 1 - \frac{1}{p^{2e}}\ell_R\left(\frac{R}{\la x^{p^e}, y^{p^e}, f^a\ra}\right)\\
  &= 1 - \frac{1}{4p^{2e}}(4rap^e - r^2a^2 + \delta^2)\\
  &= \frac{r^2}{4}\left(\frac{a}{p^e}\right)^2 - r\left(\frac{a}{p^e}\right) + 1 - \frac{\delta^2}{4p^{2e}}
\end{align*}
Extending $s(R, f^{a/p^e})$ to $[0, \infty)$ we get that
\[
  s(R, f^t) = \frac{r^2}{4}t^2 - rt + 1 - \left(\frac{\delta_f(t)}{2}\right)^2
\]
and Monsky's upper bound for $\delta(x^{p^e}, y^{p^e}, f^a)$ shows that as $p \to \infty$, $\delta_f(t) \to 0$ for $t < c$ and so the F-signature 
converges uniformly to $\frac{r^2}{4}t^2 - rt + 1$ on the interval $[0, c]$.
\end{proof}

\begin{quest} Is there some geometric significance to the quadratic polynomial to which the F-signature converges with respect to resolution of 
singularities?\end{quest}

To finish this section, we use the above method to compute the limiting quadratic polynomial of the F-signature for three distinct lines in the 
plane, which we may assume is given by $f = xy(x+y)$. Furthermore, we can compute not only the exact value of $c$ in characteristic $2 \mod 3$, but
 also the left derivative of this function at $c$ in this case. Because we show that the denominator of $c$ is always divisible by $p$ when 
$p \equiv 2 \mod 3$, the results of the previous section regarding approximate higher-order left derivatives do not apply. 

\begin{exa} Let $f = xy(x+y)$ and suppose that $k$ has characteristic $p\ge 5$ congruent to $2 \mod 3$. Define 
$b_e = \left(p - \frac{p+1}{3}\right)p^{e-1}$ so that $b_e/p^e = \frac{2}{3} - \frac{1}{3p}$. A straightforward calculation using Lucas' theorem 
shows that for all $e \in \NN$, we have that $f^{b_e} \in \m\fpow{e}$ but $f^{b_e - 1} \not\in \m\fpow{e}$. Therefore,
\begin{align*}
  \ell_R\big(R/(\m\fpow{e}: f^{b_e})\big) &= 0\\
  \ell_R\big(R/(\m\fpow{e}: f^{b_e - 1})\big) &\ne 0
\end{align*}
and so we conclude that the F-pure threshold of $xy(x+y)$ is $c = b_e/p^e = \frac{2}{3}-\frac{1}{3p}$ in this case. The above-proven limiting 
polynomial of $s(R, f^t)$ is $g(t) = \frac{9}{4}t^2 - 3t + 1$ and $g(c) = \frac{1}{4p^2}$. Because $s(R, f^c) = 0$ but $g(c) \ne 0$, we can 
directly compute $\delta_f(c)$:
\begin{align*}
  \left(\frac{\delta_f(c)}{2}\right)^2 &= g(c) - s(R, f^c)\\
  &= \frac{1}{4p^2}
\end{align*}
so then $\delta_f(c) = \frac{1}{p} = \frac{(3-2)p^{e-1}}{p^e}$ so by Monsky's bound, we have that $\delta_f(c)$ achieves a local maximum at $c$.
\end{exa}

\begin{exa} We can also provide an affirmative answer to what the left derivative is at $FPT(f) = c$ in this case. Note that the denominator is 
divisible by $p$, so we cannot apply Theorem \ref{DerivativeThm} from the previous section. We return again to the work of Teixeira, who proves 
(\cite{Teix10}, Theorem II) that if $\delta_f$ achieves a local maximum at $u \in [0, 1]$ then for all $t \in [0, 1]$ such that 
$3|t - u| \le \delta_f(u)$, we have that
\[
  \delta_f(t) = \delta_f(u) - 3|t - u|
\]
This implies that $\delta_f(t)$ is piecewise linear near $c$ so we can apply the techniques of calculus to take the left derivative at $c$:
\begin{align*}
  \frac{9}{4}t^2 - 3t + 1 - \left(\frac{\delta_f(t)}{2}\right)^2 &= \left(\frac{3}{2}t-1\right)^2 - \left(\frac{\delta_f(t)}{2}\right)^2\\
  &= \left(\frac{3}{2}t-1 + \frac{\delta_f(t)}{2}\right)\left(\frac{3}{2}t-1 - \frac{\delta_f(t)}{2}\right)
\end{align*}
now applying the product rule and substituting $c = \frac{2p-1}{3p}$, we have:
\begin{align*}
  D_-s(R, f^c) &= \left(\frac{3}{2}-\frac{1}{2}D_- \delta_f\left(\frac{2p-1}{3p}\right)\right)\left(\frac{3}{2}\left(\frac{2p-1}{3p}\right)-1 + \frac{1}{2}\delta_f\left(\frac{2p-1}{3p}\right)\right)\\
  & + \left(\frac{3}{2}+\frac{1}{2}D_-\delta_f\left(\frac{2p-1}{3p}\right)\right)\left(\frac{3}{2}\left(\frac{2p-1}{3p}\right)-1 - \frac{1}{2}\delta_f\left(\frac{2p-1}{3p}\right)\right)\\
  &= \left(\frac{3}{2}-\frac{1}{2}(3)\right)\left(1-\frac{1}{2p} -1 + \frac{1}{2p}\right) + \left(\frac{3}{2} + \frac{1}{2}(3)\right)\left(1-\frac{1}{2p} - 1 - \frac{1}{2p}\right)\\
  &= -\frac{3}{p}
\end{align*}
 To complete this computation, we used that $D_-\delta_f(c) = 3$. We can see this by recalling Theorem II from \cite{Teix10}, which tells us for 
 $t$ sufficiently close to $c$, we have $\delta_f(t) = \delta_f(c) - 3|t - c|$. This gives that the left derivative at $c$ is $3$.
\end{exa}

\noindent Note that this computation of $D_-s(R, f^c) = -\frac{3}{p} \ne 0$ in contrast to Theorem \ref{DerivativeThm} in the previous section, 
where it was shown that when $p$ does not divide the denominator of $c$ the left derivative $D_-s(R, f^c) = 0$.


\section{Computational Examples}
In this section, we will use the computational algebra package Macaulay2 to explicitly compute the F-signature of pairs for several polynomials 
and graph the data obtained using gnuplot. For the first example, we analyze the cusp $C$ and provide an affirmative answer for the left derivative
 of the F-signature at $FPT(C)$ in characteristics 5, 11, and 17. We also explicitly compute the F-signature function for three and four distinct
 linear forms in various characteristics using Macaulay2, and graph them with the quadratic limiting polynomials for the F-signature functions in
 these cases. All examples rely on routines defined in the next section.

\begin{exa} {\em (The cusp in characteristic 5, 11, and 17)} Let $C = y^2 - x^3$ be the cuspidal cubic. \end{exa}
It is known that whenever characteristic $p \equiv 2 \mod 3$, the F-pure threshold of $C$ is $\frac{5}{6} - \frac{1}{6p}$. If $6$ divides $p+1$ we
 define $b_e = (p - \frac{p+1}{6})p^{e-1}$ so that $b_e/p^e = \frac{5}{6} - \frac{1}{6p}$. In this example, we will use Macaulay2 to compute 
$s(R, C^{(b_e-1)/p^e})$ in characteristics $p = 5, 11, $ and $17$ for $e = 2$ and $3$ using a function defined in section 5 of this paper. The code 
below will compute the value of the function for $p = 5$ and $e = 2$ (so that $b_e-1 = 19$); by changing the base ring, value of $e$, and $b_e$ 
appropriately, we may use this same code for other $p$ and $e$ to obtain the corresponding values. Here {\tt Fsig} is a Macaulay2 routine defined 
explicitly in the next section; it computes $s(R, C^{a/p^e})$ for a single value of $a/p^e$.

\begin{framed}
\begin{tabular}{l}
	{\tt:R = ZZ/5[x,y]}\\
	{\tt:C = y\verb|^|2-x\verb|^|3}\\
	{\tt:Fsig(2, 19, C)}
\end{tabular}
\end{framed}

\noindent The following data was collected using the above code, changing parameters as mentioned above:

\begin{center}
\begin{tabular}{| l | l | l |}

	\hline
	$p$ & $e$ & $s(R, C^{(b_e-1)/p^e})$\\
	\hline
	5 & 2 & 1/125\\
	5 & 3 & 1/625\\
	\hline
	11& 2 & 1/1331\\
	11& 3 & 1/14641\\
	\hline
	17& 2 & 1/4913\\
	17& 3 & 1/83521\\
	\hline
	
\end{tabular}
\end{center}

\noindent So we have that for each of these values, $s(R, C^{(b_e-1)/p^e}) = 1/p^{e+1}$. Using this data, we can compute the derivative of the cusp 
$C$ in these characteristics. Let $p$ be either 5, 11, or 17 and $e$ be 2 or 3. We compute the difference quotient
\[
   \frac{s(R, C^{(b_e-1)/p^e}) - s(R, C^{b_e/p^e})}{(b_e-1)/p^e - b_e/p^e}
\]
and arrive at $-1/p$ for each value of $p$ and $e$. For these computations, notice $b_e/p^e = \frac{5}{6} - \frac{1}{6p}$ so $s(R, C^{b_e/p^e}) = 0$.
This shows that the points  
  $\left(\frac{b_2-1}{p^2}, \frac{1}{p^3}\right)$,
  $\left(\frac{b_3-1}{p^3} , \frac{1}{p^4}\right)$,
  and $\left(\frac{5}{6} - \frac{1}{6p} , 0\right)$
are colinear points on the convex function $s(R, C^t)$. By convexity of $s(R, f^t)$ we have that the F-signature is linear on the interval 
$\left[\frac{b_2-1}{p^2}, \frac{5}{6} - \frac{1}{6p}\right]$. Therefore, we can affirmatively say that the derivative of $s(R, C^t)$ at 
$t = FPT(C)$ is $-1/p$.

\begin{exa}{\em (Four distinct lines in characteristic 29)} Let $k = \mathbb{Z}/29\mathbb{Z}$ and consider $f = xy(x+y)(x+2y) \in k[x,y]$.\end{exa}

\noindent We will use the Macaulay2 functions defined in the next section to generate a graph of the F-signature of $(R, f^t)$ for $0 \le t \le 
\frac{1}{2}$. Here, {\tt GenPlot} is a function that computes the F-signature of $(R, f^t)$ at values of $t$ of the form $0 \le b/p^e \le FPT(f)$ 
for some fixed value of $e$, passed as the first argument.

\begin{framed}
\begin{tabular}{l}
  {\tt:R = ZZ/29[x,y]}\\
  {\tt:f = x*y*(x+y)*(x+2*y)}\\
  {\tt:GenPlot(2, f, "$\sim$/c29e2")}
\end{tabular}
\end{framed}

\noindent Once complete, this operation will compute the length
\[
   1- \frac{1}{29^4}\ell_R\left(\frac{R}{(x^{29^2}, y^{29^2}, f^a)}\right)
\]
for $0 \le a \le 421$ and output these lengths to a file named {\tt c29e2} (so titled for ``characteristic 29, e=2'') which is formatted to be 
graphed by the program gnuplot. We provide these two graphs of the computed F-signature and the limiting polynomial $g(t)$ here. Even at such a 
low characteristic, the two are nearly indistinguishable if plotted simultaneously.

\begin{figure}[h]
\begin{multicols}{2}
\includegraphics[scale = 0.7]{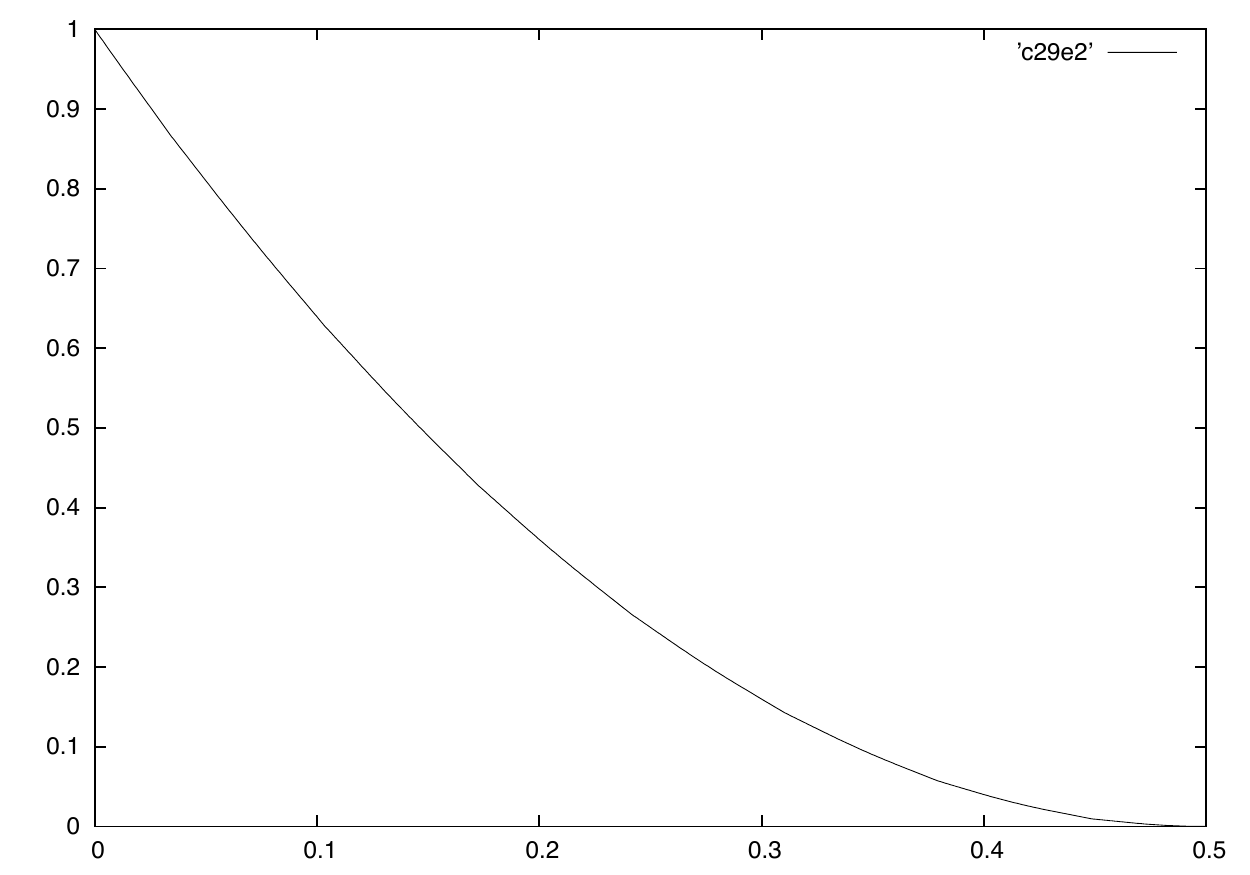}

\includegraphics[scale = 0.7]{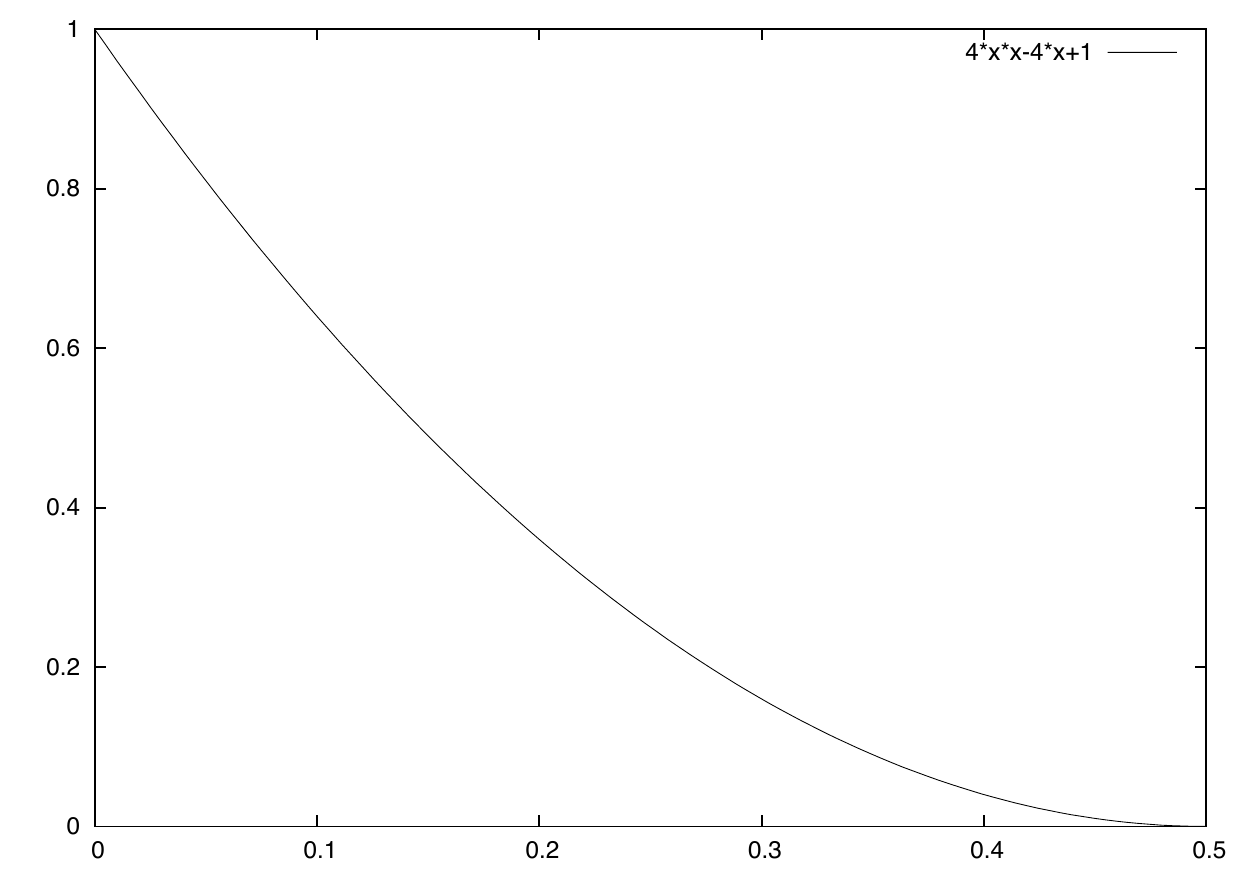}
\end{multicols}
\caption{Let $f = xy(x+y)(x+2y)$ as in example 4. The left picture here is the plot of $s(R, f^t)$ for $0 \le t \le \frac{1}{2}$ generated using 
Macaulay2. The right picture is the plot of $4t^2-4t+1$ on the same interval.}
\end{figure}

\begin{exa} {\em (Three distinct lines in characteristic 5)} Let $f = xy(x+y) \in k[x,y]$ where $k = \mathbb{Z}/5\mathbb{Z}$.\end{exa}  

\noindent Notice that we are in the case of the example at the end of the last section: $f = xy(x+y)$ and characteristic $5 \equiv 2 \mod 3$, so 
we can compute explicitly that the F-pure threshold is $\frac{2}{3}-\frac{1}{15} = \frac{3}{5}$. Using code similar to the above example, we 
generate a plot for the F-signature of $f$, the limiting polynomial, and also provide a plot of $\frac{1}{4}\delta_f(t)^2$ on $[0, \frac{3}{5}]$.

\begin{figure}[ht!]
\begin{multicols}{2}
\includegraphics[scale=0.7]{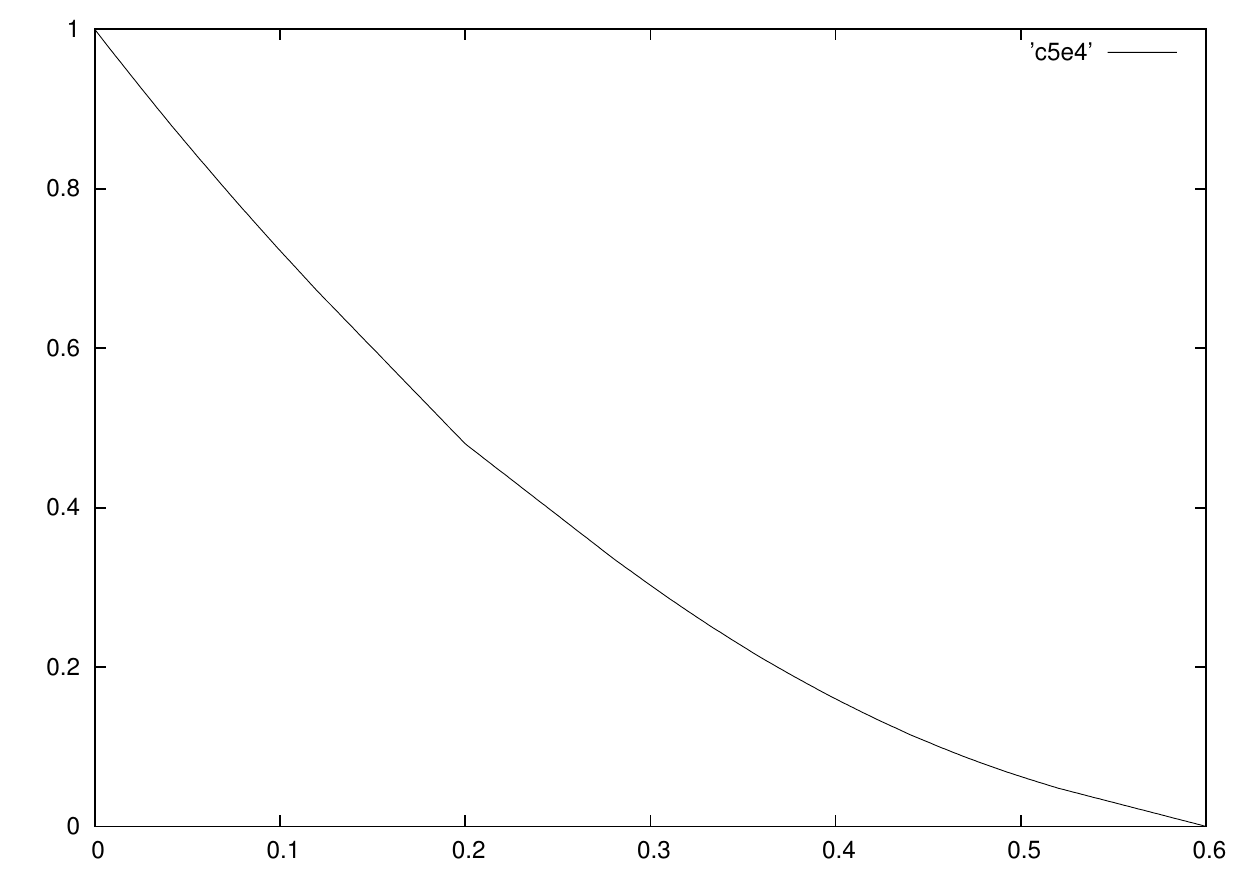}

\includegraphics[scale=0.7]{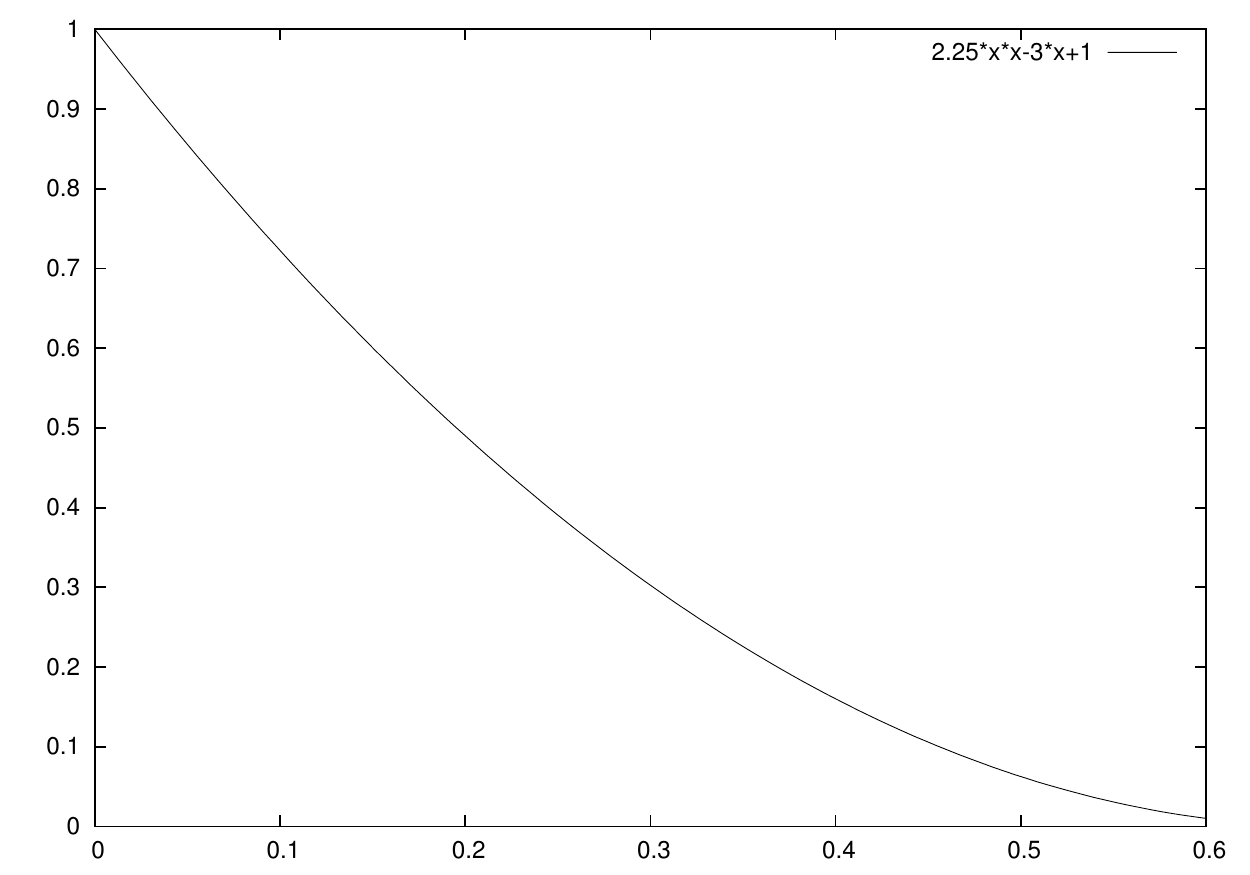}
\end{multicols}
\caption{Let $f = xy(x+y)$ as in example 5. The left picture here is the plot of $s(R, f^t)$ for $0\le t \le \frac{3}{5}$ generated using Macaulay.
 The right picture is the plot of $\frac{9}{4}t^2-3t+1$ on the same interval.}
\end{figure}

\begin{figure}[ht!]
\includegraphics{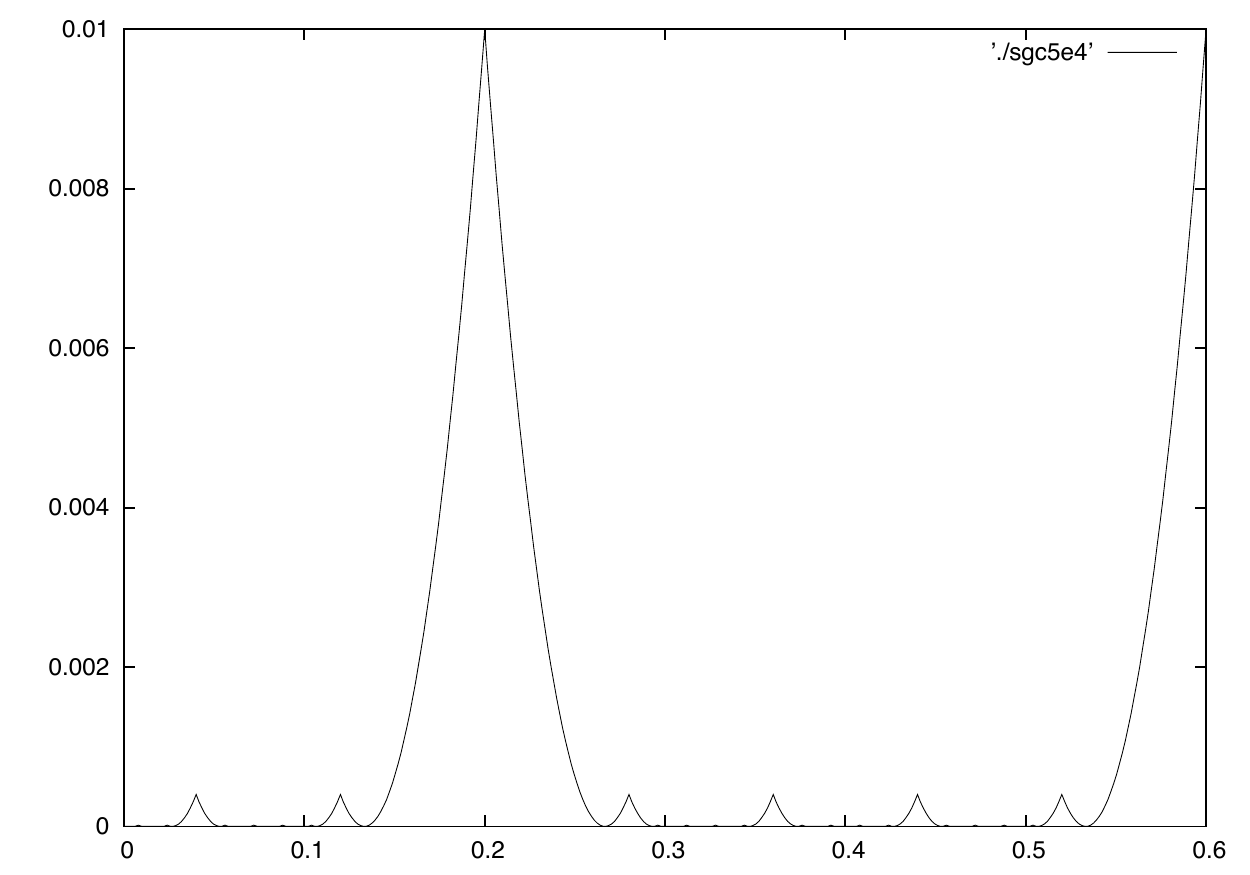}

\caption{This is the plot of the term $\frac{1}{4}\delta_f(t)^2$ on the interval $[0, \frac{3}{5}]$ which was obtained by computing $\frac{9}{4}t^2
-3t+1 - s(R, f^t)$ with $f = xy(x+y)$ as in example 5.}
\end{figure}
\vfill


\newpage

\vspace{36cm}
\section{Routines for Computing F-signatures using Macaulay2}
The results presented here were significantly influenced by experimental data gathered using the computational algebra package Macaulay2 
\cite{M2}. This final section provides the source code for functions referenced in the examples from the previous section. The first function 
defined here accepts an ideal $I$ in a polynomial ring $R$ and returns the $e^{th}$ Frobenius power $I\fpow{e}$.

\begin{framed}
{\tt
\begin{tabular}{l}
fpow = (I, e) ->\\
(\\
  L:= first entries gens I;\\
  p := char ring I;\\
\\
  J:= ideal(L\#0\verb|^|(p\verb|^|e));\\
  for i from 1 to (length L)-1 do\\
  J = J + ideal(L\#i\verb|^|(p\verb|^|e));\\
  J\\
)
\end{tabular}
}
\end{framed}

\noindent This second function returns a single length $1 - \dfrac{1}{p^{ed}}\ell_R\big(R/(x_1^{p^e}, \dots, x_d^{p^e}, f^a)\big)$, where $R$ is a 
polynomial ring in variables $x_1, \dots, x_d$ and $f$ is some polynomial in this ring.

\begin{framed}
{\tt
\begin{tabular}{l}
Fsig = (e, a, f) ->\\ (\\
 R1:=ring f;\\
 p:= char ring f;\\
 I = fpow(ideal(first entries vars R1), e);\\
     1-(1/p\verb|^|(dim(R1)*e))*degree(I+ideal(f\verb|^|a))\\ 
 )\\
\end{tabular}
}
\end{framed}

\noindent We can now build on this function to compute the F-signature of specific polynomials and output these lengths to a file. The first 
function will compute the values of the F-signature for some homogeneous polynomial $f$ (specified as the second argument when the function is 
called) at each value $a/p^e$ ($e$ is specified as the first argument) such that $0 \le a/p^e \le FPT(f)$. This is accomplished by repeatedly 
calling {\tt Fsig(e, a, f)}. The values computed are then written to a file named $fileN$ (the third argument passed to the function) which 
should be enclosed in quotation marks and give the full path name of the file. The data is stored in the correct format for use with the program 
gnuplot to produce images like those found in the previous section and a new window is opened which contains a plot of the data just computed.

\begin{framed}
{\tt
\begin{tabular}{l}
GenPlot = (e, f, fileN) ->\\
(\\
cL = for i from 0 to (char (ring f))\verb|^|e list\\
q := Fsig(e, i, f) do (stdio<<i<<",  "<<q<<endl<<"============="<<endl;\\ if q==0 then break;)\\
\\
fp = toString(fileN)<<\verb|" "|;\\
for i from 0 to (length cL)-1 do\\
fp<<toRR(i/(char (ring f))\verb|^|e)<<\verb|" "|<<toRR(cL\#i)<<endl;\\
fp<<close;\\
\\
fp="plotComm"<<"plot '"<<toString(fileN)<<"' with lines";\\
fp<<close;\\
\\
run "gnuplot -p plotComm";\\
run "rm plotComm";\\
)\\
\end{tabular}\\
}
\end{framed}

\noindent You can find this code on my website: {\tt www.math.unl.edu/$\sim$ecanton2/}.

\bibliographystyle{amsalpha}
\bibliography{refspaper2}

\begin{thebibliography}{10}

\bibitem{AE03}
Ian~M. Aberbach and Florian Enescu.
\newblock The structure of {F}-pure rings, 2005.

\bibitem{BST11}
M.~{Blickle}, K.~{Schwede}, and K.~{Tucker}.
\newblock {F-signature of pairs and the asymptotic behavior of Frobenius
  splittings}, July 2011.

\bibitem{BST112}
M.~{Blickle}, K.~{Schwede}, and K.~{Tucker}.
\newblock {F-signature of pairs: Continuity, p-fractals and minimal log
  discrepancies}, November 2011.

\bibitem{Fedd83}
Richard Fedder.
\newblock F-purity and rational singularity, 1983.

\bibitem{M2}
Daniel~R. Grayson and Michael~E. Stillman.
\newblock Macaulay2, a software system for research in algebraic geometry.
\newblock Available at http://www.math.uiuc.edu/Macaulay2/.

\bibitem{HL04}
Craig Huneke and Graham~J. Leuschke.
\newblock Two theorems about maximal cohen-macaulay modules, 2002.
\newblock 10.1007/s00208-002-0343-3.

\bibitem{Mon06}
Paul Monsky.
\newblock Mason's theorem and syzygy gaps, 2006.

\bibitem{Sch08}
Karl Schwede.
\newblock Centers of {$F$}-purity, 2010.

\bibitem{Teix02}
Pedro {Teixeira}.
\newblock p-fractals and hilbert-kunz series, 2002.

\bibitem{Teix10}
Pedro Teixeira.
\newblock Syzygy gap fractals i. some structural results and an upper bound,
  2012.

\bibitem{FsigExists}
Kevin Tucker.
\newblock F-signature exists.
\newblock 10.1007/s00222-012-0389-0.

\end{thebibliography}
\end{document}